\documentclass[11pt]{iopart}
\usepackage{amsthm,amssymb,bbm,graphicx,psfrag,color,enumerate}
\usepackage{algorithm}
\usepackage{algorithmic}
\usepackage{graphicx}
\usepackage{yfonts}

\newtheorem{thm}{Theorem}[section]

\newtheorem{lemma}[thm]{Lemma}

\newtheorem{remark}[thm]{Remark}

\def\ll{L^{\infty}(\Omega)}
\def\lll {L^2_{\diamond}(\partial \Omega)}
\def\llll{\mathcal{L}(L^2_{\diamond}(\partial \Omega))}

\def\ho{H^1_{\diamond}(\Omega)}

\newcommand{\R}{\ensuremath{\mathbbm{R}}}

\newcommand{\kommentar}[1]{}

\usepackage{tikz}
\usepackage{tikz-3dplot}

\newcommand\copyrighttext{%
  \footnotesize 
  This is an author-created, un-copyedited version of an article accepted for publication in Inverse Problems.
  The publisher is not responsible for any errors or omissions in this version of the manuscript or any version derived from it.}
\newcommand\copyrightnotice{%
\begin{tikzpicture}[remember picture,overlay]
\node[anchor=south,yshift=680pt] at (current page.south) {\fbox{\parbox{\dimexpr\textwidth-\fboxsep-\fboxrule\relax}{\copyrighttext}}};
\end{tikzpicture}%
}

\begin{document}

\copyrightnotice

\title[Enhancing residual-based techniques]{Enhancing residual-based techniques with shape reconstruction features in Electrical Impedance Tomography}

\author{Bastian Harrach${}^\dag$ and Mach Nguyet Minh${}^\dag$}
\address{\dag\ Department of Mathematics,
Goethe University Frankfurt, Germany}
\ead{\mailto{harrach@math.uni-frankfurt.de}, \mailto{mach@math.uni-frankfurt.de}}

\begin{abstract} In electrical impedance tomography, algorithms based on minimizing a linearized residual functional have been widely used due to their flexibility and good performance in practice. However, no rigorous convergence results are available in the literature yet, and reconstructions tend to contain ringing artifacts. In this work, we shall minimize the linearized residual functional under a linear constraint defined by a monotonicity test, which plays the role of special regularizer. Global convergence is then established to guarantee that this method is stable under the effects of noise. Moreover, numerical results show that this method yields good shape reconstructions under high levels of noise without the appearance of artifacts.
\end{abstract}


\ams{
35R30, 
35J25 
}



\section{Introduction}
\label{Sec:intro}
Electrical impedance tomography (EIT) is a non-invasive imaging technique which aims at reconstructing the inner structure of a reference subject from the knowledge of the current and voltage measurements on the boundary of the subject. Typically, an array of electrodes are attached to the boundary of the reference subject, then small currents are applied to some or all of the electrodes and the resulting electric voltages are measured at the electrodes. These current and voltage measurements on the boundary of the reference subject are used to estimate the value of the conductivity inside the subject. The result is an image of the inner structure of the subject due to the fact that different materials have different conductivities. Compared with computerized X-ray tomography, EIT is less costly and requires no ionizing radiation; hence, it qualifies for many clinical applications including lung ventilation (e.g. \cite{frerichs2000electrical}), brain imaging (e.g. \cite{tidswell2001three}), breast cancer diagnosis (e.g. \cite{cherepenin20013d}), etc. On the other hand, EIT can also be used for nonclinical purposes such as determining the location of mineral deposits (e.g. \cite{parker1984inverse}), describing soil structure (e.g. \cite{zhou2001three}), identifying cracks in non-destructive testing (e.g. \cite{lazarovitch2002experimental}) and so on.

The inverse problem of reconstructing the conductivity from voltage-current-measurements is known to be highly ill-posed and nonlinear, and reconstructions suffer from an enormous sensitivity to modeling and measurement errors. To reduce modeling errors, one usually concentrates on reconstructing a conductivity change with respect to a known reference conductivity. Then, the most natural approach is to parametrize the support of the conductivity change and determine these parameters by an iterative, nonlinear inverse problems solver. Such iterative methods yield good reconstructions for a given good initial guess; however, they require expensive computation and have no convergence results. Non-iterative methods such as the Factorization Method (\cite{scherzer2011handbook,harrach2013recent}) and the Monotonicity-based Method (\cite{tamburrino2002new,tamburrino2006monotonicity,harrach2013monotonicity,aykroyd2006conditional}), on the other hand, are rigorously justified and require no initial guess. However, the reconstructions of both Factorization Method and Monotonicity-based Method tend to be rather sensitive to measurement errors when phantom data or real data are applied \cite{azzouz2007factorization,harrach2010factorization,zhou2015monotonicity,garde2015convergence}.

In clinical studies, it is common practice to start by linearizing the EIT problem around a known reference conductivity and minimizing
the residual functional of the linearized equation (herein called linearized residual functional for the sake of brevity). The resulting problem can be regularized in various ways. For example, \cite{cheney1990noser,adler2009greit}  considered the minimization problem of a linearized residual functional with a regularization term, which is similar to the standard Tikhonov regularization method. The algorithm proposed in \cite{cheney1990noser,adler2009greit} is fast and simple. However, no convergence proofs are available so far.  {{Besides, one can also use sparsity reconstruction \cite{jin2012reconstruction, gehre2012sparsity,jin2012sparsity}. This is an effective method to detect piecewise constant inclusions. However, strictly speaking, convergence to the true conductivity is still an open issue, as it is not clear how to obtain a global minimizer.
		}}

Our new method is also based on minimizing the linearized residual functional. However, instead of adding a regularization term, we employ a linear constraint defined by the monotonicity test \cite[Theorem 4.1]{harrach2013monotonicity} which plays the role of a special regularizer. Global convergence of the shape reconstructions is then proved and numerical results show that this method provides good shape reconstructions even under high levels of noise. To the authors' knowledge, this is the first reconstruction method based on minimizing the residual that has a rigorous global convergence property. For the question of globally convergent algorithms for other classes of inverse problems, see for example \cite{thanh2014reconstruction}.

The paper is organized as follows. Section \ref{Sec:pre} presents some preliminaries and notations. In section \ref{Sec:theo}, we state and prove our theoretical results. Section \ref{Sec:num} shows the numerical experiments and we conclude with a brief discussion in section \ref{Sec:con}.

\section{Preliminaries and notations}
\label{Sec:pre}

We consider a bounded domain $\Omega$ in $\mathbb{R}^n$ {{(i.e. $\Omega$ is a bounded open connected subset of $\mathbb{R}^n$)}}, $n \geq 2,$ with smooth boundary $\partial \Omega$ and outer normal vector $\nu$. We assume that $\Omega$ is isotropic so that the electric conductivity $\sigma: \Omega \to \mathbb{R}$ is a scalar function, and that $\sigma$ is bounded and strictly positive. Inside $\Omega$, the electric potential $u:\Omega \to \mathbb{R}$ is governed by the so-called {\it conductivity equation}. On the boundary of $\Omega$, $u$ satisfies the {\it Neumann condition}.
\begin{eqnarray}\label{eqn:conductivity}
\nabla \cdot \sigma  \nabla u = 0 \; \mbox{ in } \Omega, \quad \sigma \partial_{\nu} u = g \; \mbox{ on } \partial \Omega.
\end{eqnarray}
In this work, we shall follow the {\it Continuum Model} (e.g. \cite[Subsection 12.3]{MuSa12}), which assumes that there are no electrodes and that the current density $g: \partial \Omega \to \mathbb{R}$ is applied over all $\partial \Omega$. The electric voltage, in this case, can be measured at every point of the boundary $\partial \Omega$ and is denoted by  $u|_{\partial \Omega}$.

The forward problem of (\ref{eqn:conductivity}) is to determine the potential $u$ for given data $\sigma$ and $g$. The existence of a variational solution $u \in H^1_{\diamond}(\Omega)$ for this Neumann boundary value problem is obtained due to the Lax-Milgram Theorem, while the uniqueness (up to an additive constant) is straight-forward. The forward problem of (\ref{eqn:conductivity}) is well-posed in the sense that the potential $u$ depends continuously on the Neumann data $g$. To guarantee that $u$ is uniquely defined, one would require furthermore that the solution $u$ has zero integral mean, i.e. $\int_{\partial \Omega} u \, {\rm d}s = 0.$

The unique solvability of the forward problem (\ref{eqn:conductivity}) guarantees the existence of the Neumann-to-Dirichlet (NtD) operator $\Lambda(\sigma)$, which maps each current $g$ to the voltage measurement $u^{\sigma}_g|_{\partial \Omega}$ on the boundary:
\[\mbox{ For each } \; \sigma \in L^{\infty}_+(\Omega), \;\Lambda(\sigma)\;:\; g \in L^2_\diamond(\partial\Omega)\to u^{\sigma}_g\in L^2_\diamond(\partial\Omega).\]
Here, $u^{\sigma}_g \in H^1_{\diamond}(\Omega)$ is the unique variational solution of the forward problem (\ref{eqn:conductivity}) corresponding to the conductivity $\sigma$ and the boundary current $g$, and $u^{\sigma}_g|_{\partial \Omega}$ is understood as the trace of $u^{\sigma}_g$ on the boundary $\partial \Omega$. $L^{\infty}_+(\Omega)$ is the subspace of $\ll$ with positive essential infima. $H^1_{\diamond}(\Omega)$ and $\lll$ denote the spaces of $H^1$- and $L^2$-functions with vanishing integral mean on $\partial \Omega$.

It is well-known that $\Lambda(\sigma)$ is a linear, bounded, compact, self-adjoint, positive operator from $\lll$ to $\lll$ (see for example \cite[Chapter 5]{Kirsch_book11} for two-dimensional cases). For each $g\in \lll$ the quadratic form associated with $\Lambda(\sigma)$ is:
\[\left<g, \Lambda(\sigma)g\right> = \int_{\partial \Omega} g\Lambda(\sigma) g\; {\textrm d}s = \int_{\Omega} \sigma |\nabla u^{\sigma}_g|^2 \; {\textrm d}x.\]
The existence of the Fr\'echet derivative $\Lambda^{\prime}$ of the NtD operator $\Lambda$ can be found for example in \cite{Lec08}. Given some direction $\kappa \in L^{\infty}(\Omega)$, the derivative $\Lambda'(\sigma)\kappa \in \llll$ is associated with the quadratic form:
\[\left<g,(\Lambda'(\sigma)\kappa)g\right>= - \int_{\Omega} \kappa |\nabla u^{\sigma}_g|^2 \; {\textrm d}x.\]

The inverse problem of (\ref{eqn:conductivity}) is to determine the conductivity $\sigma$ from the knowledge of the NtD operator $\Lambda(\sigma)$. Obviously, $\Lambda(\sigma)$ depends on $\sigma$ nonlinearly, and like many other nonlinear inverse problems, this is an ill-posed problem. In fact, it is known that small amounts of noise or model errors may cause poor spatial resolution. The reader is referred to Mueller and Siltanen's book \cite{MuSa12} for further explanation about the nonlinearity and the ill-posedness of this inverse problem.

The uniqueness of solutions of the inverse problem (\ref{eqn:conductivity}) has been investigated for different classes of conductivities and dimensions immediately following Calderon's pioneer paper \cite{cal80} in 1980, for example, Kohn and Vogelius \cite{Koh84} for piecewise analytic conductivities, Sylvester and Uhlmann \cite{Syl87} for $C^2$-conductivities in dimension $n \ge 3$, Nachman \cite{Nac96} for $W^{2,p}$-conductivities in dimension $n=2$, Astala and P\"aiv\"arinta \cite{astala2006calderon} for $L^{\infty}$-conductivities in dimension $n=2$. 

In the next section, we shall propose a regularization scheme to construct an approximate solution and prove stability in the presence of noise.



\section{Theoretical results}
\label{Sec:theo}

In this work, we need the definiteness assumption, i.e. either $\sigma\ge \sigma_0$ a.e. $\Omega$ or $\sigma\le\sigma_0$ a.e. $\Omega$. For the sake of simplicity, we shall assume that the background conductivity $\sigma_0\equiv 1$ and that the conductivity of the investigated subject is defined by ${\sigma} := 1 + \gamma\chi_D$, where the open set $D$ denotes the unknown inclusions. We assume furthermore that $\gamma \in L^{\infty}_+(D)$ and that $\overline{D} \subset \Omega$ has a connected complement. The goal of EIT is to determine the inclusions' shape $D$ from the knowledge of the NtD operators $\Lambda(\sigma)$ and $\Lambda(1)$. Notice that our method also works for inhomogeneous background conductivity.


\subsection{Exact data}
\label{subsec:exact}

We start by describing our method for exact data. The idea of this method is inspired by a result of Seo and one of the authors in \cite{harrach2010exact}. It is proved that, if $\kappa$ is an exact solution of the linearized equation
\[\Lambda'(1)\kappa = \Lambda(\sigma) -\Lambda(1)\]
then the support of $\kappa$ coincides with $D$. However, it is not clear in general whether such an exact solution exists. In addition, one cannot get a similar result for noisy data. It is natural to ask whether minimizing the residual functional 
\[r(\kappa):= \Lambda({\sigma})-\Lambda(1)-\Lambda'(1)\kappa\]
under appropriate regularization can yield a solution $\kappa$ with correct support.  In this work, we can prove that this is indeed possible. More precisely, denote by $\{\bar{g}_1,\dots,\bar{g}_N\}$ the $N$ given injected currents which are assumed to form an orthonormal subset of $\lll$, we can replace $r(\kappa)$ by the matrix $\left(\left<\bar{g}_i, r(\kappa)\bar{g}_j\right>\right)_{i,j=1}^N$ and minimize this $N$-by-$N$ matrix under the Frobenius norm: 
\begin{eqnarray}\label{29.10.1}
\min_{\kappa \in \mathcal{A} \subset L^{\infty}(\Omega)} \|{\rm \bf R}(\kappa)\|_F,
\end{eqnarray}
where ${\rm \bf R}(\kappa)$ stands for the matrix $\left(\left<\bar{g}_i, r(\kappa)\bar{g}_j\right>\right)_{i,j=1}^N$, $\|\cdot\|_F$ denotes the Frobenius norm and $\mathcal{A}$ is an admissible set of conductivity change $\kappa$. Since there is no hope to reconstruct the conductivity change at every point inside $\Omega$ from the knowledge of a finite number of measurements, it is reasonable to restrict $\mathcal{A}$ to the class of piecewise constant functions. More precisely, the admissible conductivity change $\kappa$ is assumed to be constant on a fixed partition $\{P_k\}_{k=1}^P$ of the bounded domain $\Omega$:
\[\kappa(x) = \sum_{k=1}^P a_k \chi_{P_k}(x),\]
where the $a_k$'s are real constants and each $P_k$ is assumed to be open, $\cup_{k=1}^P \overline{P}_k = \overline{\Omega}$, $\Omega \setminus P_k$ is connected and $P_k \cap P_l = \emptyset$ for $k \ne l$. Notice that, this partition is not unique, and any minimizer ${\bar{\kappa}}$ of (\ref{29.10.1}) as well as any reconstruction shape (that is, the support of some minimizer $\bar{\kappa}$) depend heavily on the choice of this partition. 

\begin{remark}
	
It is well-known that $r(\kappa)$ is a linear, bounded, compact, self-adjoint operator from $\lll$ to $\lll$. Perhaps a reasonable choice of an appropriate norm to minimize $r(\kappa)$ is the operator norm. In fact, all of the following theoretical results remain true for the operator norm. The numerical results for the operator norm can be easily obtained by considering the equivalent problem of minimizing the maximum eigenvalue of an approximate matrix of $r(\kappa)$. 

Another commonly used norm is the Hilbert-Schmidt norm. Nevertheless, we could not minimize $r(\kappa)$ under the Hilbert-Schmidt norm, since it is not clear whether or not $r(\kappa)$ belongs to the class of Hilbert-Schmidt operators. 

The idea of using the Frobenius norm comes from the fact that in realistic models, one always applies a finite number of currents $g$ on the boundary; and hence, only a finite number of measurements $\Lambda(\sigma)g$ are known.

\end{remark}

Problem (\ref{29.10.1}) was actually considered decades ago in clinical EIT such as \cite{cheney1990noser}. Typically, one usually adds a regularization term into the minimization functional, similar to the standard Tikhonov regularization method. By this method, good shape reconstructions with real-time implementation can be obtained. However, no rigorous convergence results have been established so far, and the reconstructions usually contain ringing artifacts.

In the present paper, we do not follow the standard Tikhonow regularization method. Instead, we use a linear constraint defined by the monotonicity test \cite{harrach2013monotonicity} as a special regularizer. 

\subsection*{A linear constraint defined by the monotonicity test}

A lower bound for an admissible conductivity change $\kappa$ is, in fact, due to the fact that $\sigma\ge 1$. An upper bound for $\kappa$, on the contrary, is numerically defined by the idea of the monotonicity test \cite{harrach2013monotonicity} as follows:

\cite[Example 4.4]{harrach2013monotonicity} has proved that, for $\tilde{\sigma}=1+\chi_{{D}}$ and for every ball $B$
\[B \subseteq {D} \quad \mbox{ if and only if } \quad \Lambda(1)+ \frac{1}{2}\Lambda'(1)\chi_B \ge \Lambda(\tilde{\sigma}).\]
We will show in the proof of Lemma \ref{le:position} that for any real constant ${\it a}$ satisfying $0<{\it a}\le1-\frac{1}{1+\inf_D \gamma}$, it holds that
\begin{eqnarray*}\label{28.10.1} 
   \fl \mbox{(i)\,  If } P_k \subseteq D \mbox{ then }\Lambda(1)-\Lambda(\sigma)+\alpha \Lambda'(1)\chi_{P_k}\geq 0\; \mbox{ in quadratic sense for (at least) all } \alpha\in [0,{\it a}]. \hfill\\
  \fl \mbox{(ii) If } P_k \nsubseteq D \mbox{ then } \Lambda(1)-\Lambda(\sigma)+\alpha \Lambda'(1)\chi_{P_k}\not\geq 0\; \mbox{  in quadratic sense for all } \alpha> 0. \hfill
\end{eqnarray*}
We see that, when $\gamma \equiv 1$, we have ${\rm \it a}=\frac{1}{2}$ and in this case, $\rm \it a$ is actually the number $\frac{1}{2}$ in \cite[Example 4.4]{harrach2013monotonicity}. Although the formula of $\rm \it a$ depends heavily on the inclusion contrast $\gamma$, in many applications a bound for  $\gamma$ is known a-priori.

For each pixel $P_k$, the biggest coefficient $\beta_k$ such that
\begin{eqnarray}\label{ineq:1}
\Lambda(1) - \Lambda(\sigma) + \alpha \Lambda^{\prime}(1) \chi_{P_k} \ge 0  \quad \forall \alpha \in [0,\beta_k]
\end{eqnarray}
then satisfies: 
\begin{itemize}
\item $\beta_k \ge {\it a} >0$ if $P_k \subseteq D$.
\item $\beta_k=0$ if $P_k \nsubseteq D$.
\end{itemize}
This motivates the constraint $0 \le \kappa \le \beta_k$ on each pixel $P_k$. Note that $\beta_k$ is allowed to be $\infty$. Our following theory, therefore, requires us to use a stronger upper bound $\min({\it a},\beta_k)$ where $\rm\it a$ plays the role of a special regularizer. Since ${\it a}$ is smaller than the true contrast $\gamma$, this seems ``over-constrained", but we will show that the minimizer of this over-constrained problem possesses the correct support. Thus, we can define the admissible set $\mathcal{A}$ as follows:
\[\mathcal{A}:=\left\{ \kappa \in L^{\infty}(\Omega)\;:\; \kappa = \sum_{k=1}^P a_k\chi_{P_k}, a_k \in \mathbb{R}, 0 \le a_k \le \min({\it a},\beta_k) \right\}.\]

Here comes our main result.

\begin{thm}\label{thm:main}
Consider the minimization problem
\begin{eqnarray}\label{5.11.15.1}
\min_{\kappa \in \mathcal{A}} \|{\rm \bf R}(\kappa)\|_F.
\end{eqnarray}
The following statements hold true:
\begin{itemize}
\item[(i)] Problem (\ref{5.11.15.1}) admits a unique minimizer $\hat{\kappa}$.
\item[(ii)] ${\rm supp}\, \hat{\kappa}$ and $D$ agree up to the pixel partition, i.e. for any pixel $P_k$ 
\[P_k \subset {\rm supp}\,\hat{\kappa} \; \mbox{  if and only if  } \; P_k \subset D.\] 
Moreover, $\hat{\kappa} = \sum_{k=1}^P \min ({\it a},\beta_k) \chi_{P_k}$.
\end{itemize}
\end{thm}

\begin{remark}\label{re10.11.15} 
\begin{itemize}
\item[(i)] Notice that $\beta_k$ is defined via the infinite-dimensional NtD operator $\Lambda$ and does not involve the finite-dimensional matrix $\rm\bf R$.

{{\item[(ii)] Theorem \ref{thm:main} holds regardless the number $N$ of applied boundary currents. } }

\item[(iii)] We would like to emphasize that the goal of our method is to show an approximation of the size and the location of the inclusion $D$, not the value of the conductivity $\sigma$ inside $D$. Indeed, as we can see from the above theorem, the support of the unique minimizer $\hat{\kappa}$ agrees with $D$ up to pixels, while the value of $\hat{\kappa}$ is always smaller than $\sigma-\sigma_0$.
\item[(iv)] \cite[Theorem 4.3]{Lec08} has proved that $\Lambda'(1)$ is injective if there are enough boundary currents (that is, if $N$ is sufficiently large). In that case, it is obvious to see that ${\rm \bf R}(\kappa)$ is also injective. This fact together with the fact that the square Frobenius norm is strictly convex imply $\kappa \mapsto \|{\rm \bf R}(\kappa)\|^2_F$ is strictly convex.
\end{itemize}
\end{remark}

Before proving Theorem \ref{thm:main}, we need the following lemmas.
\begin{lemma} \label{le:position} For any pixel $P_k$, $P_k \subseteq D$ if and only if $\beta_k>0$.

\end{lemma}
One special case of Lemma \ref{le:position} has been proved in \cite[Example 4.4]{harrach2013monotonicity}. We have slightly modified the proof there to fit with our notations and settings.


\begin{proof}[Proof of Lemma \ref{le:position}] {\bf Step 1:} We shall check that, $P_k \subseteq D$ implies $\beta_k>0$. Indeed, employing the following monotonicity principle (see e.g. \cite[Lemma 3.1]{harrach2013monotonicity})
\[\left< g, (\Lambda(\sigma_2)-\Lambda(\sigma_1))g \right> \ge \int_{\Omega} \frac{\sigma_2}{\sigma_1} (\sigma_1 - \sigma_2)|\nabla u_2|^2\,{\rm d}x\]
for $\sigma_1:=\sigma$ and $\sigma_2:=1$, we get the following inequalities for all pixels $P_k$, all $\alpha \in [0,{\it a}]$ and all $g\in \lll$:
\begin{eqnarray*} 
\fl \left< g, (\Lambda({\sigma}) - \Lambda(1)-\Lambda'(1)\alpha \chi_{P_k})g\right> \le -\int_{\Omega} \left(1-\frac{1}{{\sigma}} \right) |\nabla u^0_{g}|^2\,{\rm d}x +\int_{\Omega}\alpha\chi_{P_k}|\nabla u^0_g|^2 \le  \\
 \le -\int_{D} \left(1-\frac{1}{{\sigma}} \right) |\nabla u^0_{g}|^2\,{\rm d}x +\int_{P_k} {\it a} \,|\nabla u^0_g|^2\,{\rm d}x \le 0.  
\end{eqnarray*}
Here $u^0_g$ is the unique solution of the forward problem (\ref{eqn:conductivity}) when the conductivity is chosen to be $1$. The last inequality holds due to the fact that ${\it a}\le1-\frac{1}{1+\inf_D \gamma}\le 1-\frac{1}{\sigma}$ in $D$ and that $P_k$ lies inside $D$. 

{\bf Step 2:} If $\beta_k>0$, we will show that $P_k \subseteq D$ by contradiction: Assume that $\beta_k>0$ and $P_k \nsubseteq D$. Applying the following monotonicity principle \cite[Lemma 3.1]{harrach2013monotonicity}
\[\Lambda(\sigma_1)-\Lambda(\sigma_2) \ge \Lambda'(\sigma_2)(\sigma_1 - \sigma_2),\]
where we choose $\sigma_1:=\sigma, \sigma_2:=1$, and taking into account the definition of $\beta_k$ in (\ref{ineq:1}), we obtain
\begin{eqnarray*}
0\ge \Lambda({\sigma}) - \Lambda(1)- \Lambda'(1)\beta_k \chi_{P_k} \ge \Lambda'(1)({\sigma}-1)-\Lambda'(1)\beta_k\chi_{P_k}.
\end{eqnarray*}
Thus, for all $g\in \lll$:
\begin{eqnarray}\label{ineq:2}
\int_{P_k}\beta_k|\nabla u^0_g|^2\,{\rm d}x \le \int_{\Omega} ({\sigma}-1) |\nabla u^0_g|^2 \,{\rm d}x \le  \int_{D} C |\nabla u^0_g|^2 \,{\rm d}x
\end{eqnarray}
with some positive constant $C$.

On the other hand, applying localized potential (see e.g. \cite[Theorem 3.6]{harrach2013monotonicity} and \cite{Geb08} for the origin of this idea), we can find a sequence $\{g_m\} \subset \lll$ such that the solutions $\{u^0_m\} \subset \ho$ of the forward problem (\ref{eqn:conductivity}) (when the conductivity is chosen to be $1$ and boundary currents $g=g_m$)
fulfill
\[\lim_{m \to \infty}\int_{P_k} |\nabla u^0_m|^2\,{\rm d}x = \infty, \quad \mbox{ and } \quad \lim_{m \to \infty}\int_{D} |\nabla u^0_m|^2\,{\rm d}x =0.\]
This contradicts (\ref{ineq:2}). 
\end{proof}

\begin{lemma}\label{Sk.positive} For all pixels $P_k$, denote by ${\rm \bf S}_k$ the matrix $(-\left<\bar{g}_i,\Lambda^{\prime}(1)\chi_{P_k}\bar{g}_j\right>)_{i,j=1}^N$. Then ${\rm \bf S}_k$ is a positive definite matrix.
\end{lemma}
 
 \begin{proof}[Proof of Lemma \ref{Sk.positive}] For all ${\rm x}=(x_1,\dots,x_N)^\top \in \mathbb{R}^N$, we have
\begin{eqnarray*}
\fl{\rm x}^{\top} {\rm \bf S}_k\, {\rm x} = -\sum_{i,j=1}^N x_ix_j \left<\bar{g}_i,\Lambda'(1)\chi_{P_k} \bar{g}_j \right> = -\left<\bar{g}, \Lambda'(1)\chi_{P_k} \bar{g}\right> = \int_{P_k}|\nabla u^0_{\bar{g}}|^2\;{\textrm d}x \ge 0,
\end{eqnarray*}
where $\bar{g}=\sum_{i=1}^N x_i\bar{g}_i$. This means that ${\rm \bf S}_k$ is a positive semi-definite symmetric matrix in $\mathbb{R}^{N \times N}$. We shall prove that ${\rm \bf S}_k$ is, in fact, a positive definite matrix by showing that
\begin{eqnarray}\label{6.11.15.1}\int_{P_k} |\nabla u^0_g|^2\;{\textrm d}x >0 \quad \mbox{ for all } \; g \in \lll,\; \|g\|_{\lll} \ne 0.
\end{eqnarray}
Assuming, by contradiction, that there exists $\tilde{g} \in \lll, \|\tilde{g}\|_{\lll} \ne 0$ such that 
\[\int_{P_k} |\nabla u^0_{\tilde{g}}|^2\;{\textrm d}x =0.\]
Since $P_k$ is open, there exists an open ball $B \subseteq P_k$. It holds that $u^0_{\tilde{g}}=c$ \; a.e. $B$, where $c$ is some real constant. This fact can be obtained in many different ways, for example, by the Poincar\'e's inequality.

On the other hand, $u(x)=c$ is a solution of the forward problem (\ref{eqn:conductivity}) with conductivity $1$ and homogeneous Neumann boundary data. Therefore, $u^0_{\tilde{g}}-c$ is a solution of (\ref{eqn:conductivity}) when the conductivity is $1$ and the Neumann boundary is $\tilde{g}$. Moreover, we have proved that $u^0_{\tilde{g}}-c=0$ \; a.e. in an open ball $B \subset \Omega$, the unique continuation principle implies that $u^0_{\tilde{g}}-c=0$ \; a.e. $\Omega$, and hence $\tilde{g}=0$ a.e. $\partial \Omega$. This contradiction implies that (\ref{6.11.15.1}) holds.
 \end{proof}

Now we are in the position to prove our main result.
\begin{proof}[Proof of Theorem \ref{thm:main}]
{\bf (i)} Existence of minimizer: Since the functional 
\[\kappa \mapsto \|{\rm \bf R}(\kappa)\|^2_F:= \sum_{i,j=1}^N \left<\bar{g}_i, r(\kappa)\bar{g}_j \right>^2\] 
is continuous, it admits a minimizer in the compact set $\mathcal{A}$. Uniqueness obviously follows when we have proven (ii).

{\bf (ii) Step 1:} Denote by $\chi_i:=\chi_{P_i}$. We shall check that, for all $\kappa=\sum_{k=1}^P \alpha_k \chi_k$ satisfying $0 \le \alpha_k \le \min ({\it a},\beta_k)$, it holds that $r(\kappa) \le 0$ in quadratic sense.

Notice that ${\it a} \le 1 - \frac{1}{\sigma}$ in $D$. In the same manner as the proof of Lemma \ref{le:position}, we can write
\begin{eqnarray*}
\left< g, (\Lambda({\sigma}) - \Lambda(1)-\Lambda'(1)\kappa)g\right> \le -\int_{D} {\it a} |\nabla u^0_{g}|^2\,{\rm d}x +\sum_{k=1}^P \int_{P_k}\alpha_k|\nabla u^0_g|^2\,{\rm d}x 
\end{eqnarray*}
for any $g \in \lll$. 

If $\alpha_k >0$, we have $\beta_k \ge \alpha_k >0$. By Lemma \ref{le:position}, it holds that $P_k \subseteq D$. Taking into account that  $\alpha_k \le {\it a}$ and that $P_i \cap P_j \ne \emptyset$ for $i \ne j$, we get that $\left<g,r(\kappa)g \right> \le 0$ for any $g \in \lll$.

{\bf Step 2:} Let $\hat{\kappa}=\sum_{k=1}^P \hat{\alpha}_k \chi_k$ be a minimizer of (\ref{5.11.15.1}). We prove that ${\rm supp}\hat{\kappa} \subseteq D$. 

Indeed, if $\hat{\alpha}_k>0$, it follows from Step $1$ and the monotonicity of $\Lambda'(1)$ that 
\[\Lambda({\sigma})-\Lambda(1)-\Lambda'(1){\hat{\alpha}_k}\chi_k \le \Lambda({\sigma})-\Lambda(1)-\Lambda'(1)\hat{\kappa} \le 0.\]
This implies $\beta_k>0$. Thanks to Lemma \ref{le:position}, we have $P_k \subseteq D$. 

{\bf Step 3:} If $\hat{\kappa}=\sum_{k=1}^P \hat{\alpha}_k \chi_k$ is a minimizer of (\ref{5.11.15.1}), then $\hat{\kappa}=\sum_{k=1}^P \min ({\it a},\beta_k) \chi_k.$

Indeed, it holds that $\hat{\kappa} \le \sum_{k=1}^P \min ({\it a},\beta_k) \chi_k$. If there exists a pixel $P_k$ such that $\hat{\kappa}(x) < \min\{ {\it a},\beta_k\}$ in $P_k$, we can choose $h>0$ such that $\hat{\kappa} + h \chi_k = \min ({\it a},\beta_k)  \mbox{ in } P_k.$
 We will show that
\begin{eqnarray}\label{7.11.15.11}
\|{\rm \bf R}(\hat{\kappa}+h\chi_{k})\|_F <  \|{\rm \bf R} (\hat{\kappa})\|_F
\end{eqnarray}
which contradicts the minimality of $\hat{\kappa}$. 

Let $\lambda_1(\hat{\kappa}) \ge \lambda_2(\hat{\kappa}) \ge \dots \ge \lambda_N(\hat{\kappa})$ be $N$ eigenvalues of ${\rm \bf R}(\hat{\kappa})$ and $\lambda_1(\hat{\kappa}+ h \chi_k) \ge \lambda_2(\hat{\kappa}+ h \chi_k) \ge \dots \ge \lambda_N(\hat{\kappa}+ h \chi_k)$ be $N$ eigenvalues of ${\rm \bf R}(\hat{\kappa}+ h \chi_k)$. Since ${\rm \bf R}(\hat{\kappa})$ and ${\rm \bf R}(\hat{\kappa}+ h \chi_k)$ are both symmetric, all of their eigenvalues are real. By the definition of the Frobenius norm, we get
\begin{eqnarray}\label{7.11.15.12}
&&\|{\rm \bf R}(\hat{\kappa} + h \chi_k)\|^2_F - \|{\rm \bf R}(\hat{\kappa})\|^2_F =\sum_{i=1}^N |\lambda_i(\hat{\kappa} + h \chi_k)|^2 - \sum_{i=1}^N |\lambda_i(\hat{\kappa})|^2= \nonumber\\
&&= \sum_{i=1}^N \left( \lambda_i(\hat{\kappa} + h \chi_k) + \lambda_i(\hat{\kappa})\right) \cdot \left( \lambda_i(\hat{\kappa} + h \chi_k) - \lambda_i(\hat{\kappa})\right). 
\end{eqnarray}
Thanks to Step 1, $r(\hat{\kappa}) \le 0$ and $r(\hat{\kappa} + h \chi_k) \le 0$ in the quadratic sense. Thus, for all ${\rm x}=(x_1,\dots,x_N)^\top \in \mathbb{R}^N$, we have
\begin{eqnarray*}
{\rm x}^{\top} {\rm \bf R}(\hat{\kappa})\, {\rm x} = \sum_{i,j=1}^N x_ix_j \left<\bar{g}_i,r(\hat{\kappa})\bar{g}_j \right> = \left<\bar{g},r(\hat{\kappa}) \bar{g}\right> \le 0,
\end{eqnarray*}
where $\bar{g}=\sum_{i=1}^N x_i\bar{g}_i$. This means that $-{\rm \bf R}(\hat{\kappa})$ is a positive semi-definite symmetric matrix in $\mathbb{R}^{N \times N}$. It is well-known that all eigenvalues of a positive semi-definite symmetric matrix should be non-negative. Thus,
\begin{eqnarray}\label{cond:1} \quad \lambda_i(\hat{\kappa}) \le 0  \quad \mbox{ for all }\quad i \in \{ 1,\dots, N\}.
\end{eqnarray}
In the same manner, $-{\rm \bf R}(\hat{\kappa}+h\chi_k)$ is also a positive semi-definite symmetric matrix. Hence,
\begin{eqnarray}\label{cond:2} \lambda_i(\hat{\kappa} + h \chi_k) \le 0 \quad \mbox{ for all }\quad i \in \{ 1,\dots, N\}.
\end{eqnarray}
On the other hand, Lemma \ref{Sk.positive} claims that ${\rm \bf S}_k$ is a positive definite matrix. Thus, all $N$ eigenvalues $\lambda_1({\rm \bf S}_k) \ge \dots \ge \lambda_N({\rm \bf S}_k)$ of ${\rm \bf S}_k$ are positive. Since
\[{\rm \bf R}(\hat{\kappa}+h\chi_{k}) = {\rm \bf R}(\hat{\kappa}) + h {\rm \bf S}_k.\]
and the matrices ${\rm \bf R}(\hat{\kappa}+h\chi_{k}), {\rm \bf R}(\hat{\kappa})$ and ${\rm \bf S}_k$ are all symmetric, we can apply Weyl's Inequalities \cite[Theorem III.2.1]{bhatia1997matrix} to get
\begin{eqnarray} \label{cond:3}
\lambda_i(\hat{\kappa}+h\chi_{k}) \ge \lambda_i(\hat{\kappa}) + h \lambda_N({\rm \bf S}_k) > \lambda_i(\hat{\kappa}) \quad \mbox{ for all } \quad i\in \{1,\dots,N\}. 
\end{eqnarray}
In summary, (\ref{7.11.15.12}), (\ref{cond:1}), (\ref{cond:2}) and (\ref{cond:3}) imply (\ref{7.11.15.11}). This ends the proof of Step 3.

{\bf Step 4:} If $P_k \subseteq D$, then $P_k \subseteq {\rm supp}\hat{\kappa}$. Indeed, since $\hat{\kappa}$ is a minimizer of (\ref{5.11.15.1}), Step 3 implies that  $\hat{\kappa}=\sum_{k=1}^P \min ({\it a},\beta_k) \chi_k.$ Since $P_k \subseteq D$, it follows from Lemma \ref{le:position} that $\min({\it a},\beta_k)>0$. Thus, $P_k \subseteq {\rm supp}\hat{\kappa}$.

In conclusion, problem (\ref{5.11.15.1}) admits a unique minimizer $\hat{\kappa} = \sum_{k=1}^P \min ({\it a},\beta_k) \chi_k$. This minimizer fulfills 
\begin{equation*}
\hat \kappa= \cases{ {\it a} \quad \mbox{ in } \; P_k, &if  $P_k$ lies inside $D$,\\
0  \quad \mbox{ in } \; P_k,&if  $P_k$ does not lie inside $D$.}
\end{equation*}
\end{proof}


\subsection{Convergence for noisy data}
\label{subsec:noisy}

In the presence of noise, we denote by $\delta$ the noise level. Similarly as above, we call ${\rm \bf R}^{\delta}(\kappa)$ the $N$-by-$N$ matrix $\left( \left< \bar{g}_i,r^{\delta}(\kappa)\bar{g}_j\right>\right)_{i,j=1}^N$, where the residual for noisy data now reads
\[r^{\delta}(\kappa):=\Lambda^{\delta}(\sigma)-\Lambda(1) - \Lambda^{\prime}(1)\kappa,\]
and the error is bounded from above by $\delta$ in the operator norm, i.e. 
\[\|\Lambda^{\delta}(\sigma) - \Lambda(\sigma) \|_{\llll} \le \delta.\]
When we replace the exact data $\Lambda(\sigma)$ by the noisy data $\Lambda^{\delta}(\sigma)$, we have to change the definition of the biggest coefficient $\beta_k$, too. To this end, we need the following lemma:

\begin{lemma}\label{le:positive} For any bounded linear operator $T$ on a real Hilbert space $H$:
\begin{itemize} 
\item[(i)] (Square Root Lemma) If $T$ is positive, i.e. $T$ is self-adjoint and $\left<x,Tx\right> \ge 0$ for all $x \in H$, then there exists a unique bounded linear positive operator $U$ on $H$ such that $U^2 = T$. Moreover, $U$ commutes with every bounded linear operator which commutes with $T$. We call $U$ the positive square root of $T$ and denote by $U = \sqrt{T}$.
\item[(ii)] (Absolute value of a bounded linear operator) The modulus of $T$
\[|T|:= \sqrt{T^*T}\]
is a positive operator, where $T^*$ is the adjoint operator of $T$. Moreover, $|T|$ commutes with every bounded linear operator which commutes with $T^*T$.

\item[(iii)] (Positive decomposition) If $T$ is self-adjoint, then there exists a unique pair of bounded positive operators $T_+$ and $T_-$ such that $T=T_+ - T_-$, $T_+T_-=0$, $T_+$ and $T_-$ commute with each other and with $T$. Moreover, $|T|=T_+ + T_-$. 
\item[(iv)] If $T$ is positive, then $|T|=T$. 
\item[(v)] If $T$ is self-adjoint, then $T \le |T|$ in quadratic sense.
\item[(vi)] For any bounded linear operators $A,B$ on a real Hilbert space $H$:
\begin{eqnarray}\label{30.10.20}
\||A|-|B|\|_{\mathcal{L}(H)}^2 \le \left(\|A\|_{\mathcal{L}(H)}+\|B\|_{\mathcal{L}(H)} \right)\|A - B\|_{\mathcal{L}(H)}.
\end{eqnarray}
Consequently, the absolute value of operators is continuous w.r.t. the operator norm.
\end{itemize}
\end{lemma}
\begin{proof}[Proof of Lemma \ref{le:positive}]
\begin{itemize}
\item[(i)] This famous fact can be proved either by means of classical functional analysis (see, for example, \cite[Appendix]{young1988introduction}) or by using $C^*$-algebra techniques (\cite[Proposition 4.33]{douglas2012banach}).
\item[(ii)] For all $x\in H$, it holds that
\[\fl\left<x,T^*Tx\right>=\left<Tx,Tx\right> \ge 0 \quad \mbox{ and } \quad \left<x,T^*Tx\right>=\left<Tx,Tx\right> = \left<T^*Tx,x\right>.\]
Hence, $T^*T$ is a positive operator. Thanks to (i), $T^*T$ has a unique positive square root $\sqrt{T^*T}$, which commutes with every bounded linear operator commuting with $T^*T$.

\item[(iii)] We shall prove this fact by using techniques in $C^*$-algebra. Indeed, since $T$ is self-adjoint, it is normal. Hence, the Gelfand map $\Gamma$ establishes a $*$-isometric isomorphism between the space of continuous functions on $\sigma(T)$ (here $\sigma(T)$ denotes the spectrum of $T$) and the $C^*$-algebra $\textswab{C}^*$ (see, e.g. \cite[Theorem 2.31]{douglas2012banach}). Define by 
\[f_+(x)=\max(x,0) \quad \mbox{ and } \quad f_-(x)=-\min(x,0),\]
then both $f_+$ and $f_-$ are continuous positive functions on $\sigma(T)$. Thus, $T_+:=f_+(T)=\Gamma(f_+)$ and $T_-:=f_-(T)=\Gamma(f_-)$ are well-defined positive operators on the space of bounded linear operators $\mathcal{B}(H)$. Moreover, $T_+$ and $T_-$ commute with each other and with $T$. Since $f_+$ and $f_-$ satisfy
\begin{eqnarray*}
 f_+(x) - f_-(x) = x, \quad f_+(x)+f_-(x) = |x|, \quad f_+(x)f_-(x)=0
\end{eqnarray*}
for all $x$, it follows that the two positive operators $T_+$ and $T_-$ satisfy
\begin{eqnarray*}
 T_+ - T_- =T, \quad T_++ T_-=|T|, \quad T_+T_-=0.
\end{eqnarray*}

On the other hand, we also have
\[\|T_+\|_{\mathcal{L}(H)}=\sup_{x \in \sigma(T)}|f_+(x)| \le \sup_{x \in \sigma(T)}|x| = \|T\|_{\mathcal{L}(H)}.\]
By the same way, one can prove that $T_-$ is also a bounded operator. It remains to show that the pair $(T_+,T_-)$ is unique. Indeed, if $T=A-B$ where $A,B$ are positive bounded operators satisfying $AB=0$, then $T^2=A^2+B^2=(A+B)^2$. Thus $A+B$ is the unique square root of $T^2$, i.e. $A+B=|T|$. Hence, $A=(|T|+T)/2=T_+$ and $B=(|T|-T)/2=T_-$. The uniqueness holds.

\item[(iv)] Since $T$ is positive, it holds that $T^2 = T^*T$. Thanks to (i), $T$ is a unique positive square roof of $T^*T$, i.e. $T=\sqrt{T^*T}$. On the other hand, by (ii), $\sqrt{T^*T}=|T|$. Thus, $T=|T|$.



\item[(vi)] First we prove that, for any linear, bounded, positive operators $U,V$ on a real Hilbert space H, it holds that
\begin{eqnarray}\label{30.10.100}
\|U-V\|_{\mathcal{L}(H)}^2 \le \|U^2-V^2\|_{\mathcal{L}(H)}.
\end{eqnarray}
The above inequality had been proved in \cite[Theorem 1]{kittaneh1986inequalities}. We shall recite the proof here for the reader's convenience. 

Denote by $W:=U-V$ then $W$ is linear bounded self-adjoint operator on $H$. Hence
\begin{eqnarray*}
\|W\|_{\mathcal{L}(H)} = \sup_{\|g\|_{H} =1} \left| \left< g, W g \right> \right|.
\end{eqnarray*}
Thus, we can find a sequence $\{g_n\} \subset H$ such that 
\[\|g\|_{H}=1\quad \mbox{ and } \quad |\left<g_n,Wg_n\right>| \uparrow \|W\|_{\mathcal{L}(H)} \;\; \mbox{ as } n \to \infty.\] 
Denote by $t:=\|W\|_{\mathcal{L}(H)}$, we have
\begin{eqnarray}\label{30.10.1000}
\fl\|U^2-V^2\|_{\mathcal{L}(H)} &=& \|UW+WV\|_{\mathcal{L}(H)} \ge \left| \left<g_n,(UW+WV)g_n\right>\right|\nonumber\\
&=&\left| \left<g_n, U(W-t)g_n \right> + \left<(W-t)g_n,Vg_n \right> + t\left<g_n,(U+V)g_n\right> \right|\nonumber\\
&\ge&t \left<g_n,(U+V)g_n\right> - \left|\left<g_n,U(W-t)g_n\right>\right|-\left|\left<(W-t)g_n,Vg_n\right>\right|.
\end{eqnarray}
Since both $U$ and $V$ are positive operators, we have $U+V \ge \pm W$. Thus, 
\begin{eqnarray}\label{30.10.1001}
\left<g_n,(U+V)g_n\right> \ge \left| \left<g_n,Wg_n\right>\right| \quad \mbox{ for all } n.
\end{eqnarray}
On the other hand, $Wg_n \to tg_n$ as $n \to \infty$ because
\begin{eqnarray*}
\fl \|Wg_n-tg_n\|_{\mathcal{L}(H)}^2&=&\left<Wg_n-tg_n,Wg_n-tg_n\right>=\|Wg_n\|^2 - 2t\left<g_n,Wg_n\right> + t^2\\
&\le&2t^2-2t\left<g_n,Wg_n\right>.
\end{eqnarray*}
Combining (\ref{30.10.1000}), (\ref{30.10.1001}) and taking $n \to \infty$, we get (\ref{30.10.100}).

Thanks to (\ref{30.10.100}), for all bounded linear operators $A$ and $B$ we have
\begin{eqnarray*}
\fl \| |A|-|B|\|_{\mathcal{L}(H)}^2 &\le& \| |A|^2-|B|^2\|_{\mathcal{L}(H)} = \|A^*A -B^*B\|_{\mathcal{L}(H)}\\
& \le &\|A^*\|_{\mathcal{L}(H)}\|A-B\|_{\mathcal{L}(H)}+ \|A^*-B^*\|_{\mathcal{L}(H)}\|B\|_{\mathcal{L}(H)}\\
&=&\left(\|A\|_{\mathcal{L}(H)}+\|B\|_{\mathcal{L}(H)}\right)\|A-B\|_{\mathcal{L}(H)}.
\end{eqnarray*}
\end{itemize}
\end{proof}

Back to our issues, since $V:=\Lambda(1) - \Lambda(\sigma)$ is a linear bounded positive operator, Lemma \ref{le:positive} yields that $|V|=V$. The monotonicity test (\ref{ineq:1}) can be rewritten as
\begin{eqnarray}\label{eqn:monotonicitytest111}
|V| + \alpha \Lambda^{\prime}(1) \chi_{P_k} \ge 0 \quad \mbox{ for all } \quad \alpha \in [0,\beta_k].
\end{eqnarray}
When replacing the exact data $\Lambda(\sigma)$ by the noisy data $\Lambda^{\delta}(\sigma)$, the above inequality does not hold in general for all $P_k \subseteq D$. Indeed, since the operator in the left-hand side of (\ref{eqn:monotonicitytest111}) is compact, it has eigenvalues arbitrarily close to zero. A small noise will make these eigenvalues a little bit negative which can make the $\beta_k$ defined in (\ref{eqn:monotonicitytest111}) zero everywhere. Hence, we replace the test (\ref{eqn:monotonicitytest111}) with
\begin{eqnarray}\label{cond:noise}
|V^{\delta}|+ \alpha \Lambda^{\prime}(1) \chi_{k} \ge -\delta I 
\end{eqnarray}
where $V^{\delta}:= \Lambda(1) - \Lambda^{\delta}(\sigma)$ and $I$ is the identity operator from $\lll$ to $\lll$. Since we can always redefine the data $V^{\delta}$ by defining $V^{\delta}:=(V^{\delta}+{(V^{\delta}})^{*})/2$, without loss of generality, we can assume that $V^{\delta}$ is self-adjoint.

We then define $\beta^{\delta}_k$ as the biggest coefficient such that inequality (\ref{cond:noise}) holds for all $\alpha \in [0,\beta^{\delta}_k]$. We see that, $\min({\it a},\beta^{\delta}_k)$ will still be ${\it a}$ inside the inclusions but possibly a little bit larger than $0$ outside. More precisely, we have the following lemma:


\begin{lemma}\label{le:compare}
Assume that $\|\Lambda^{\delta}(\sigma) - \Lambda(\sigma)\|_{\llll} \le \delta$, then for every pixel $P_k$, it holds that $\beta_k \le \beta^{\delta}_k$ for all $\delta >0$.
\end{lemma}

\begin{proof}[Proof of Lemma \ref{le:compare}]
It is sufficient to check that $\beta_k$ satisfies (\ref{cond:noise}). Indeed, since the operator $V-V^{\delta}$ is linear, bounded and self-adjoint, we have for all $\|g\|_{\lll}=1$:

\[\left|\left< g, (V - V^{\delta})g\right> \right|  \le \|V-V^{\delta}\|_{\llll}= \|\Lambda^{\delta}(\sigma) - \Lambda(\sigma)\|_{\llll}\le \delta.\]
Now for any $\tilde{g} \in \lll$, if $\|\tilde{g}\|_{\lll} \ne 0$, then
\begin{eqnarray*}
\left|\left<{\tilde{g}},(V-V^{\delta}){\tilde{g}}\right>\right| = \|\tilde{g}\|^2\left|\left<\frac{\tilde{g}}{\|\tilde{g}\|},( V - V^{\delta})\frac{\tilde{g}}{\|\tilde{g}\|}\right>\right| \le \delta \|\tilde{g}\|^2.
\end{eqnarray*}
Thus, $V^{\delta}-V \ge -\delta I$ in quadratic sense. Besides, Lemma \ref{le:positive} implies $|V^{\delta}|\ge V^{\delta}$. Hence, 
\begin{eqnarray*}\label{eqn:11}
 |V^{\delta}| + \beta_k \Lambda^{\prime}(1) \chi_k  \ge V^{\delta} + \beta_k \Lambda^{\prime}(1) \chi_k = V + \beta_k \Lambda^{\prime}(1) \chi_k + V^{\delta} - V \ge  -\delta I.
\end{eqnarray*}

\end{proof}
As a consequence of Lemma \ref{le:compare}, it holds that
\begin{itemize}
\item[1.] If $P_k$ lies inside $D$, then $\beta_k^{\delta} \ge a$.
\item[2.] If $\beta_k^{\delta} = 0$, then $P_k$ does not lie inside $D$.
\end{itemize}

We end this section by proving the following stability result:
\begin{thm}\label{thm:stability} Consider the minimization problem
\begin{eqnarray}\label{prb:3}
\min_{\kappa \in \mathcal{A}^{\delta}} \|{\rm \bf R}^{\delta}(\kappa)\|_F
\end{eqnarray}
where ${\rm \bf R}^{\delta}(\kappa)$ represents the $N$-by-$N$ matrix $\left( \left< \bar{g}_i,r^{\delta}(\kappa)\bar{g}_j\right>\right)_{i,j=1}^N$, and the admissible set for noisy data is defined by
\[\mathcal{A}^{\delta}:=\left\{\kappa \in L^{\infty}(\Omega): \kappa = \sum_{k=1}^P a_k \chi_{P_k}, \; a_k \in \mathbb{R}, \; 0 \le a_k \le \min({\it a},\beta^{\delta}_k)\right\}.\]
The following statements hold true:
\begin{itemize}
\item[(i)] Problem (\ref{prb:3}) admits a minimizer.
\item[(ii)] Let $\hat{\kappa}:=\sum_{k=1}^P \min({\it a},\beta_k)\chi_k$ and $\hat{\kappa}^{\delta}:=\sum_{k=1}^P \hat{a}^{\delta}_k \chi_k$ be minimizers of problems (\ref{5.11.15.1}) and (\ref{prb:3}) respectively. Then $\hat{\kappa}^{\delta}$ pointwise converges to $\hat{\kappa}$ as $\delta \to 0$. 
\end{itemize}
\end{thm}

\begin{remark} Same argument as Remark \ref{re10.11.15}, we get the strict convexity of   $\kappa \mapsto \|{\rm \bf R}^{\delta}(\kappa)\|^2_F$ when the number $N$ of boundary currents are sufficiently large.  Since all minimizers of (\ref{prb:3}) minimize $\|{\rm \bf R}^{\delta}(\kappa)\|^2_F$, it holds that  (\ref{prb:3}) has a unique minimizer if $N$ is sufficiently large.
\end{remark}

\begin{proof}[Proof of Theorem \ref{thm:stability}]
{\bf (i)} The existence of minimizers of (\ref{prb:3}) is obtained in the same manner as Theorem \ref{thm:main}(i). Indeed, since the functional $\kappa \mapsto \|{\rm \bf R}^{\delta}(\kappa)\|^2_F$ is continuous, (\ref{prb:3}) admits at least one minimizer in the compact set $\mathcal{A}^{\delta}$.


{\bf (ii)}
{\bf Step 1:} Convergence of a subsequence of $\hat{\kappa}^{\delta}$

For any fixed $k$, the sequence $\{\hat{a}_k^{\delta}\}_{\delta>0}$ is bounded from below by $0$ and from above by ${\it a}$. By Weierstrass' Theorem, there exists a subsequence $(\hat{a}^{\delta_n}_1, \dots, \hat{a}^{\delta_n}_P)$ converging to some limit $(a_1,\dots,a_P)$. Of course, $0\le a_k \le {\it a}$ for all $k=1,\dots,P$.

{\bf Step 2:} Upper bound of the limit 

We shall check that $a_k \le \beta_k$ for all $k=1,\dots,P$. Indeed, thanks to (\ref{30.10.20})
\begin{eqnarray*}
 \||V^{\delta}|-|V|\|_{\llll}^2 &\le& \left(\|V^{\delta}\|_{\llll}+\|V\|_{\llll}\right)\|V^{\delta}-V\|_{\llll}\\
& \le &(2\|V\|_{\llll}+\delta)\delta.
\end{eqnarray*}
Thus, $|V^{\delta}|$ converges to $|V|$ in the operator norm as $\delta \to 0$. Hence, for any fixed $k$, 
\[ |V| + a_k \Lambda^{\prime}(1)\chi_k = \lim_{\delta_n \to 0} \left(|V^{\delta_n}| + \hat{a}^{\delta_n}_k \Lambda^{\prime}(1) \chi_k \right)\]
 in the operator norm. It is straight-forward to see that, for any $g\in \lll$
\[
\fl \left<g, (|V| + a_k \Lambda^{\prime}(1) \chi_k) g \right> = \lim_{\delta_n \to 0} \left<g,(|V^{\delta_n}| + \hat{a}^{\delta_n}_k \Lambda^{\prime}(1) \chi_k ) g \right> 
\ge - \lim_{\delta_n \to 0} \left<g,\delta_n g\right> =0.
\]

{\bf Step 3:} Minimality of the limit

By Lemma \ref{le:compare}, $\min({\it a},\beta_k) \le \min({\it a},\beta_k^{\delta})$ for all $k=1,\dots,P$. Thus, $\hat{\kappa}$ belongs to the admissible class of the minimization problem (\ref{prb:3}) for all $\delta >0$. By the minimality of $\hat{\kappa}^{\delta}$, we get
\begin{eqnarray}\label{31.10.2}
\|{\rm \bf R}^{\delta}(\hat{\kappa}^{\delta})\|_F \le \|{\rm \bf R}^{\delta}(\hat{\kappa})\|_F .
\end{eqnarray}
Denote by $\kappa = \sum_{k=1}^P a_k \chi_k$, where $a_k$'s are the limits obtained in Step 1. We have that
\[\|{\rm \bf R}^{\delta_n}(\hat{\kappa}^{\delta_n}) \|^2_F = \sum_{i,j=1}^N \left<\bar{g}_i,\left( -V^{\delta_n}-  \sum_{k=1}^P \hat{a}^{\delta_n}_k  \Lambda^{\prime} (1)\chi_k\right)\bar{g}_j\right>^2\]
and
\[\|{\rm \bf R}({\kappa}) \|^2_F = \sum_{i,j=1}^N \left<\bar{g}_i,\left( -V -  \sum_{k=1}^P {a}_k  \Lambda^{\prime} (1)\chi_k\right)\bar{g}_j\right>^2.\]
Since $V^{\delta}$ converges to $V$ in the operator norm, for all fixed $g \in \lll$, $V^{\delta}g$ converges to $Vg$ in $\lll$. Taking into account of the fact that $\hat{a}^{\delta_n}_k$ converges to $a_k$ for any $k \in \{1,\dots,P\}$ as $\delta_n \to 0$, it is easy to check that $\|{\rm \bf R}^{\delta_n}(\hat{\kappa}^{\delta_n})\|_F $ converges to $\|{\rm \bf R}(\kappa)\|_F$ as $\delta_n \to 0$. In the same manner, we can show that $\|{\rm \bf R}^{\delta_n}(\hat{\kappa})\|_F$ converges to $\|{\rm \bf R}(\hat{\kappa})\|_F$. Thus, it follows from (\ref{31.10.2}):
\[\|{\rm \bf R}(\kappa)\|_F \le \|{\rm \bf R}(\hat{\kappa})\|_F.\]
Since $\kappa$ belongs to the admissible class of problem (\ref{5.11.15.1}), the above inequality implies that it is in fact a minimizer of (\ref{5.11.15.1}). By the uniqueness of the minimizer, we obtain $\kappa = \hat{\kappa}$, that is $a_k=\min({\it a},\beta_k)$.

{\bf Step 4:} Convergence of the whole sequence $\hat{\kappa}^{\delta}$

We have proved so far that every subsequence of $(\hat{a}^{\delta}_1,\dots,\hat{a}^{\delta}_P)$ has a convergent subsubsequence, that converges to the limit $l=(\min({\it a},\beta_1),\dots,\min({\it a},\beta_P))$. This implies the convergence of the whole sequence  $(\hat{a}^{\delta}_1,\dots,\hat{a}^{\delta}_P)$ to $l$. This ends the proof of Theorem \ref{thm:stability}.
\end{proof}

\section{Numerical results}
\label{Sec:num}

We do the numerical experiment for the case $\Omega$ is the unit disk in $\mathbb{R}^2$ centered at the origin. We consider the current density $\bar{g}_i$ in the following orthonormal set of $L^2(\partial \Omega)$:
\begin{eqnarray*} \label{orthonormal}
\left\{ \frac{1}{\sqrt{\pi}} \sin (j\phi), \frac{1}{\sqrt{\pi}} \cos(j\phi) \; \left| \right.  \;j=1, \dots, N_1\right\}, \quad N=2N_1,
\end{eqnarray*}
here and in the following, we choose $N_1=16$ and $\phi$ is an angle from the positive $x$-axis. We shall follow the notations in Section \ref{Sec:theo}. Denote by ${\rm \bf V}^{\delta}$ the matrix $(\left<\bar{g}_i,V^{\delta}\bar{g}_j\right>)_{i,j=1}^N$. The minimization matrix ${\rm \bf R}^{\delta}(\kappa)$ now reads
\[{\rm \bf R}^{\delta}(\kappa)=-{\rm \bf V}^{\delta}+\sum_{k=1}^P a_k {\rm \bf S}_k,\]
and we would like to find $(a_1,\dots,a_P)$ satisfying the linear constraint $0 \le a_k \le \min({\it a},\beta_k^{\delta})$ so that $\|{\rm \bf R}^{\delta}(\kappa)\|_F$ is minimized.

First, we shall collect the known data ${\rm \bf V}^{\delta}$ and ${\rm \bf S}_k \;(k=1,\dots,P)$. Then we calculate $\beta_k^{\delta}$, and finally, minimize ${\rm \bf R}^{\delta}(\kappa)$.

\begin{figure}[h!]
\centering
\includegraphics[width=1.0\textwidth]{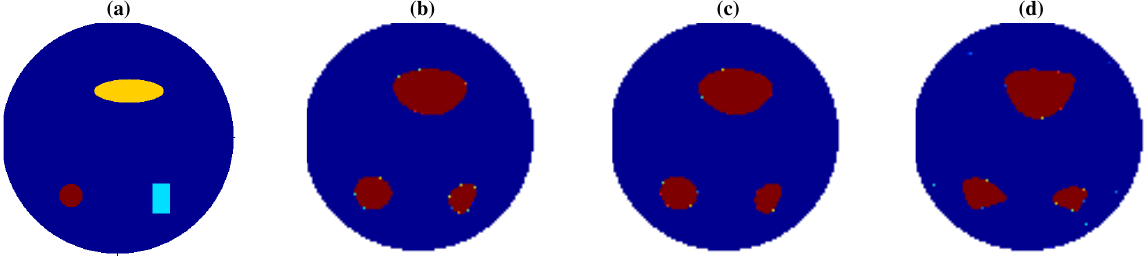}
\caption{Reconstruction of conductivity change under a linear constraint defined by the monotonicity test: (a) true distribution of conductivity change; (b)  $0.1\%$ relative noise; (c) $1\%$ relative noise; (d) $10\%$ relative noise.}
                \label{fig:1}
\end{figure}

\begin{figure}
        \centering
\includegraphics[width=1.0\textwidth]{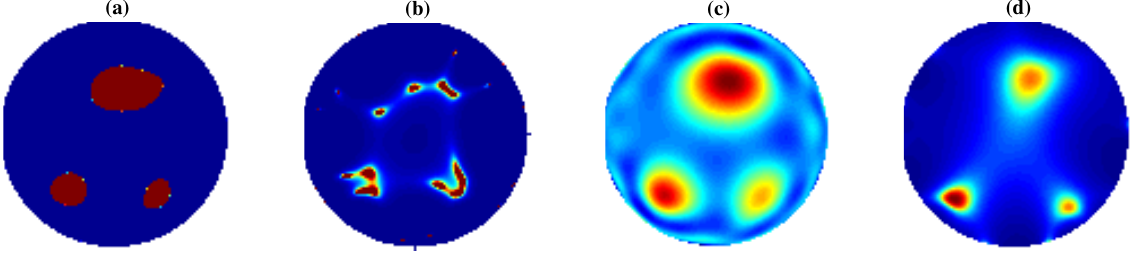}
                \caption{Reconstruction of conductivity change with $5\%$ relative noise: (a) constraint $0\le \kappa \le a$; (b) constraint $0 \le \kappa \le \sum_{k=1}^P \beta^{\delta}_k \chi_k$; (c) without constraint, regularization term $10^{-5}\|x\|_2$; (d) in the operator norm, constraint $0\le \kappa \le \sum_{k=1}^P \min({\it a},\beta^{\delta}_k)$.}
                \label{fig:2}
                  
   \end{figure}
   
   \begin{figure}
        \centering
                        \includegraphics[width=0.99\textwidth]{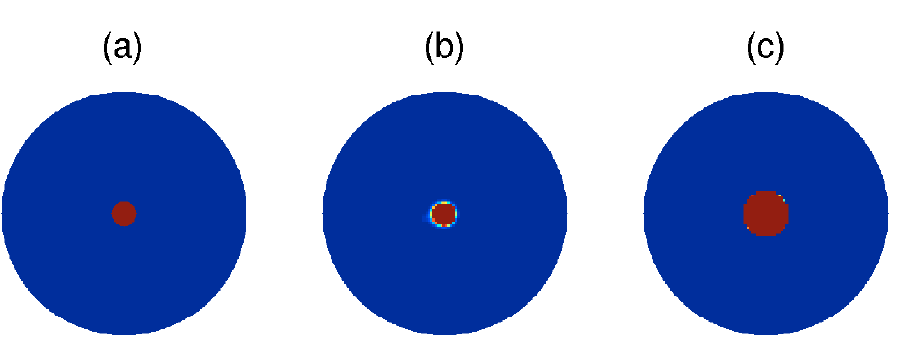}
                \caption{Reconstruction of conductivity change with $10^{-11}$ relative noise: (a) true distribution of conductivity change; (b) constraint $0\le \kappa \le \sum_{k=1}^P \min(a,\beta^{\delta}_k) \chi_k$; (c) constraint $0\le \kappa \le a$.}
                \label{fig:3}\end{figure}
                
                   \begin{figure}
        \centering
                        \includegraphics[width=0.73\textwidth]{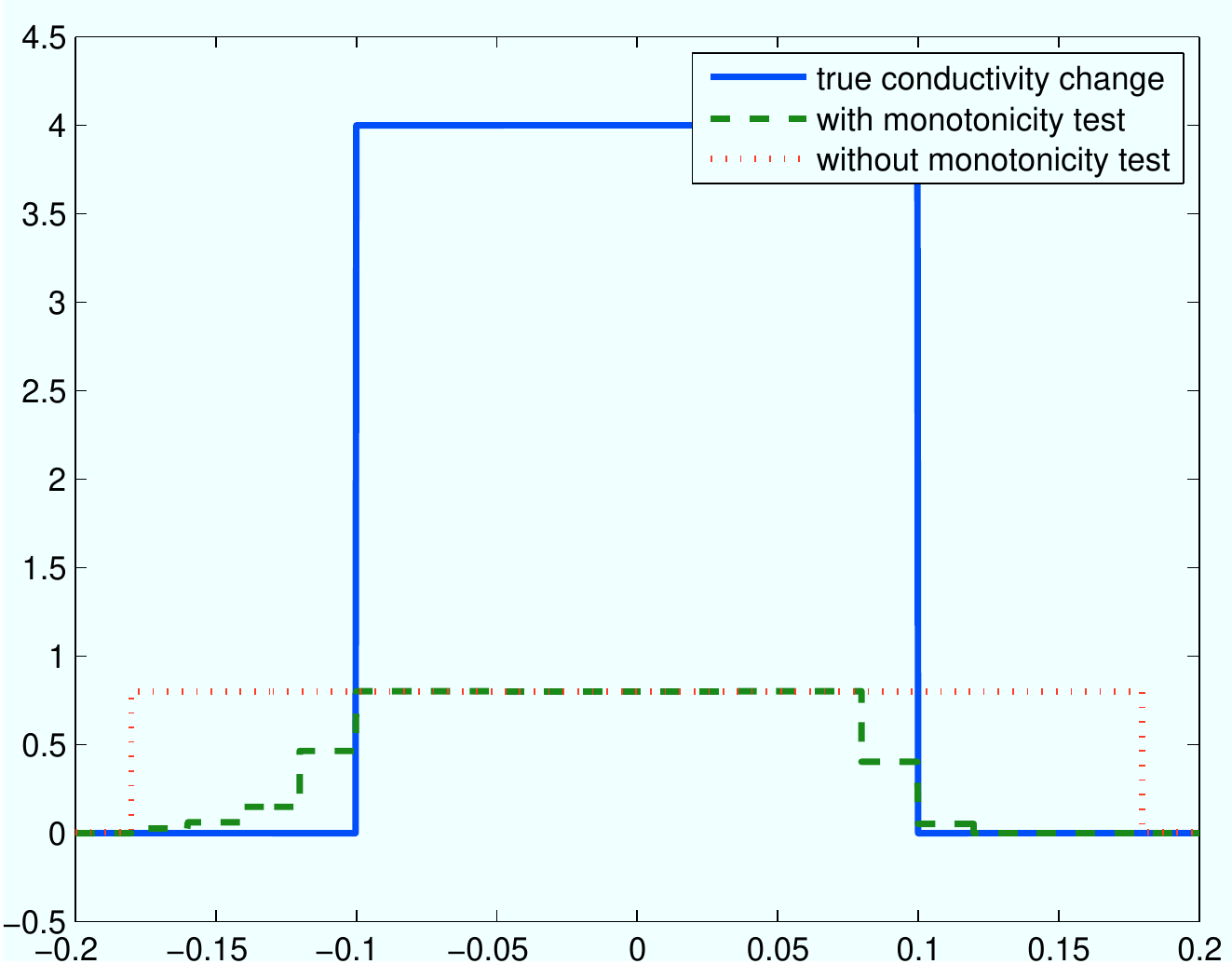}
                \caption{Cut of Figure \ref{fig:3} through the $x$-axis.}
                \label{fig:4}\end{figure}

\begin{figure}
        \centering
\includegraphics[width=1.0\textwidth]{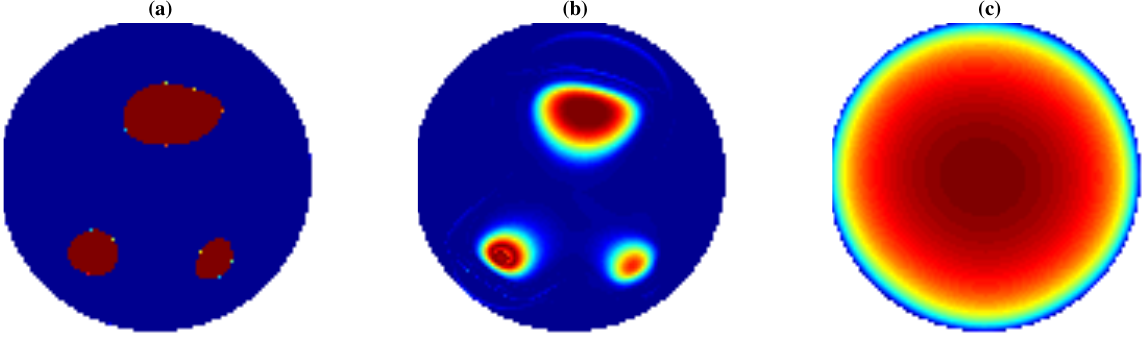}
                \caption{Reconstruction of conductivity change with $5\%$ relative noise: (a) using \texttt{cvx}; (b) using \texttt{quadprog} with
``trust-region-reflective"; (c) using \texttt{quadprog} with default option.}
                \label{fig:5}
                  
   \end{figure}

\subsection{Generating data}
\label{subsec:data}

We shall calculate ${\rm \bf V}^{\delta}$ and ${\rm \bf S}_k \; (k=1,\dots,P)$ with the help of  \texttt{COMSOL}, a commercial finite element software. We remind that ${\rm \bf V}^{\delta}$ is the difference between two matrices $(\left<\bar{g}_i,\Lambda(1)\bar{g}_j\right>)_{i,j=1}^{N}$ and $(\left<\bar{g}_i,\Lambda^{\delta}(\sigma)\bar{g}_j\right>)_{i,j=1}^{N}$; while ${\rm \bf S}_k$ is the matrix $(-\left<\bar{g}_i,\Lambda^\prime(1)\chi_k\bar{g}_j\right>)_{i,j=1}^{N}$ corresponding to pixel $P_k$.


When the conductivity is chosen to be $1$, the forward problem (\ref{eqn:conductivity}) becomes the Laplace equation with Neumann boundary condition $\bar{g}_j$, and admits a unique solution on the unit disk:
\begin{equation}
u^0_j= \cases{ \frac{1}{j\sqrt{\pi}} \sin (j \varphi) r^j, &if  $\bar{g}_j = \frac{1}{\sqrt{\pi}} \sin (j \varphi)$,\\
\frac{1}{j\sqrt{\pi}} \cos (j \varphi) r^j,&if  $\bar{g}_j = \frac{1}{\sqrt{\pi}} \cos (j \varphi)$,}
\end{equation}
here the pair $(r,\phi)$ forms the polar coordinates with respect to the center of $\Omega$. $(\left<\bar{g}_i,\Lambda(1)\bar{g}_j\right>)_{i,j=1}^{N}$ and ${\rm \bf S}_k$ are uniquely defined via:
\[ \left<\bar{g}_i,\Lambda(1)\bar{g}_j\right>=\int_{\partial \Omega}\bar{g}_iu^0_j \; \mbox{ and } \; ({\rm \bf S}_k)_{i,j}=- \left<\bar{g}_i,\Lambda^{\prime}(1)\chi_k \bar{g}_j \right> = \int_{P_k} \nabla u^0_i \cdot \nabla u^0_j. \]

Instead of calculating $(\left<\bar{g}_i,\Lambda(1)\bar{g}_j\right>)_{i,j=1}^{N}$ and $(\left<\bar{g}_i,\Lambda^{\delta}(\sigma)\bar{g}_j\right>)_{i,j=1}^{N}$ separately, we follow the method suggested in \cite{GH07}, which first uses \texttt{COMSOL} to compute ${\rm \bf V}$ and then add noise to ${\rm \bf V}$ to obtain ${\rm \bf V}^{\delta}$. We have
\[ {\rm \bf V}_{i,j}=\left<\bar{g}_i,(\Lambda(1)-\Lambda(\sigma))\bar{g}_j \right> = \int_{\partial \Omega} \bar{g}_i (u^0_j - u_j),\]
where $u_j$ is the unique solution of (\ref{eqn:conductivity}) for conductivity $\sigma$ and boundary current $\bar{g}_j$. Denote by $d_j$ the difference $u^0_j-u_j$, then $d_j$ satisfies the following system
\begin{eqnarray}\label{eqn:110} 
\nabla \cdot \left( \sigma \nabla d_j - (\sigma-1) \nabla u^0_j \right) & = & 0 \quad \mbox{ in } \Omega, \nonumber\\ 
\left( \sigma \nabla d_j - (\sigma -1) \nabla u^0_j \right) \cdot \nu &=& 0 \quad \mbox{ on } \partial \Omega. 
\end{eqnarray}
This system can be solve by using the Coefficient Form PDE model built by \texttt{COMSOL}. Notice that, we have to add the constraint $\int_{\partial \Omega} d_j = 0$ in order to guarantee the uniqueness of solution of (\ref{eqn:110}).

Under the absolute noise $\delta$, the noisy data ${\rm \bf V^{\delta}}$ can be obtained from ${\rm \bf V}$ by
\[{\rm \bf V^{\delta}}:= {\rm \bf V} + \frac{{\rm \bf E}}{\|{\rm \bf E}\|_F}\delta,\]
here ${\rm \bf E}$ is a random matrix in $\mathbb{R}^{N\times N}$ with uniformly distributed entries between $-1$ and $1$. The Hermitian property of ${\rm \bf V}^{\delta}$ follows by redefining $\quad {\rm \bf V^{\delta}}:=(({\rm \bf V}^{\delta})^* + {\rm \bf V}^{\delta})/2.$

\subsection{Finding $\beta^{\delta}_k$}
\label{subsec:beta}

{{In numerical experiment, the parameter $\beta^\delta_k$ also depends on the number of boundary currents $N$. For that reason, we shall call it $\beta_k^{\delta,N}$, where the superscript $N$ denotes the dependence of $\beta_k^{\delta,N}$ on the number of boundary currents. 
		
The infinite-dimensional operators $|V^\delta|$ and $\Lambda'(1)\chi_{P_k}$ in the formula (\ref{cond:noise}) will be replaced by the $N$-by-$N$ matrices ${\rm \bf V}^{\delta}$ and ${\rm \bf S}_k$ as follows
\begin{eqnarray}\label{eqn:numerror}
-\alpha {\rm \bf S}_k \ge -\delta {\rm \bf I} - |{\rm \bf V}^{\delta}| \quad \mbox{ for all } \quad \alpha \in [0,\beta^{\delta,N}_k].
\end{eqnarray}
where the modulus $|{\rm \bf V}^{\delta}|:=\sqrt{({\rm \bf V}^{\delta})^* {\rm \bf V}^{\delta}}$ in this case is called the absolute value of matrices. 
		
It is easy to see that $\beta_k^{\delta,N}\downarrow \beta_k^\delta$ as $N \to \infty$. We shall follow the argument in \cite{Ha15} to calculate $\beta_k^{\delta,N}$.

Since $\delta {\rm \bf I} + |{\rm \bf V}^{\delta}|$ is Hermitian positive-definite, the Cholesky decomposition allows us to decompose it into the product of a lower triangular matrix and its conjugate transpose, i.e.
\[\delta {\rm \bf I} + |{\rm \bf V}^{\delta}| = {\rm \bf L} {\rm \bf L}^*,\]
where ${\rm \bf L}$ is a lower triangle matrix with real and positive diagonal entries. This decomposition is unique. Moreover, since
\[0< {\rm det}(\delta {\rm \bf I} + {\rm \bf V}^{\delta}) = {\rm det}({\rm \bf L}) \;{\rm det} ({\rm \bf L}^*) = {\rm det}({\rm \bf L}) \;\overline{{\rm det ({\bf L})}},\]
it follows that ${\rm \bf L}$ is invertible. For each $\alpha>0$, we have that
\[-\alpha {\rm \bf S}_k + \delta {\rm \bf I} + |{\rm \bf V}^{\delta}| = -\alpha {\rm \bf S}_k + {\rm \bf LL}^*= {\rm \bf L} ( -\alpha {\rm \bf L}^{-1} {\rm \bf S}_k ({\rm \bf L^*})^{-1} + {\rm \bf I}){\rm \bf L}^*.\]
Hence, the positive semi-definiteness of $-\alpha {\rm \bf S}_k + \delta {\rm \bf I} + |{\rm \bf V}^{\delta}|$ is equivalent to the positive semi-definiteness of $-\alpha {\rm \bf L}^{-1} {\rm \bf S}_k ({\rm \bf L^*})^{-1} + {\rm \bf I}$. 

Since both $-{\rm \bf L}^{-1} {\rm \bf S}_k ({\rm \bf L^*})^{-1}$ and $-\alpha {\rm \bf L}^{-1} {\rm \bf S}_k ({\rm \bf L^*})^{-1} + {\rm \bf I}$ are Hermitian matrices, we can apply Weyl's Inequalities \cite[Theorem III.2.1]{bhatia1997matrix} to obtain
\begin{eqnarray}\label{eqn:num_eig}
\lambda_j(-\alpha {\rm \bf L}^{-1} {\rm \bf S}_k ({\rm \bf L^*})^{-1} + {\rm \bf I}) = \alpha \lambda_j(-{\rm \bf L}^{-1} {\rm \bf S}_k ({\rm \bf L^*})^{-1}) + 1,\quad j=1,\dots,N,
\end{eqnarray}
where $\lambda_1({\rm \it a}) \ge \dots \ge \lambda_N({\rm \it a})$ denote the $N$-eigenvalues of some matrix ${\rm \it a}$. 

It is known from Lemma \ref{Sk.positive} that ${\rm \bf S}_k$ is a positive definite matrix. Thus, the matrix ${\rm \bf L}^{-1}{\rm \bf S}_k ({\rm \bf L}^*)^{-1}$ is positive definite, too. This implies $\lambda_j(-{\rm \bf L}^{-1} {\rm \bf S}_k ({\rm \bf L^*})^{-1}) < 0$ for all $j=1,\dots,N$. It then follows from (\ref{eqn:num_eig}) that for every $j \in \{1,\dots,N\}$, the functional $\alpha \mapsto \lambda_j(-\alpha {\rm \bf L}^{-1} {\rm \bf S}_k ({\rm \bf L^*})^{-1} + {\rm \bf I})$ is decreasing in $\alpha$. The biggest $\alpha$ that fulfills 
\[\lambda_j(-\alpha {\rm \bf L}^{-1} {\rm \bf S}_k ({\rm \bf L^*})^{-1} + {\rm \bf I}) \ge 0 \quad \mbox{ for all } \quad j=1,\dots,N\] 
should be
\[\beta^{\delta,N}_k = -\frac{1}{\lambda_N(-{\rm \bf L}^{-1} {\rm \bf S}_k ({\rm \bf L^*})^{-1})},\]
where $\lambda_N(-{\rm \bf L}^{-1} {\rm \bf S}_k ({\rm \bf L^*})^{-1})$ is the smallest (most negative) eigenvalue of $-{\rm \bf L}^{-1} {\rm \bf S}_k ({\rm \bf L^*})^{-1}$.
} }

\subsection{Minimizing the residual}
\label{subsec:min}

{{The minimization problem (\ref{prb:3}) can be rewritten as follows.
		\begin{eqnarray}\label{prb:4}
		\min \left\{\|\sum_{k=1}^P a_k {\rm \bf S}_k -{\rm \bf V}^{\delta}\|_F: 0 \le a_k \le \min({\it a},\beta^{\delta,N}_k) \right\}.
		\end{eqnarray}
		Since the minimization functional $J: (a_1,\dots,a_P) \mapsto \|  \sum_{k=1}^P a_k {\rm \bf S}_k-{\rm \bf V}^{\delta} \|^2_F$ is convex, every minimizer should be global minimizer. The existence of the minimizer of (\ref{prb:4}) follows the continuity of the minimization functional $J$ and the fact that $\{a_k\}_{k=1}^P$ is uniformly bounded. 
		
		Let $\hat{\kappa}^{\delta,N}:=\sum_{k=1}^{N} \hat{a}_k^{\delta,N} \chi_k$, where $(\hat{a}_1^{\delta,N},\dots,\hat{a}_P^{\delta,N})$ be a minimizer of (\ref{prb:4}). When $N$ is large enough, in the same manner as Step 2, proof of Theorem \ref{thm:stability}, we can show that $\hat{a}_k^{\delta,N} \le \beta_k$ for $k=1,\dots,P$. Hence, we can follow the proof of Theorem  \ref{thm:stability} to conclude that $\hat{\kappa}^{\delta,N}$ pointwise converges to the unique minimizer $\hat{\kappa}$ of (\ref{5.11.15.1}) as $\delta$ goes to $0$ (and $N$ is large enough).
	}}

		Problem  (\ref{prb:4}) can be solved either by using \texttt{cvx}, a package for specifying and solving convex programs \cite{cvx,gb08}, or with the \texttt{MATLAB}  built-in function \texttt{quadprog}. Notice that reconstruction images are highly affected by the choice of the minimization algorithms (see Figure \ref{fig:5} and Table 1).
		
		While \texttt{cvx} allows us to directly work with matrices, \texttt{quadprog} requires us to  rearrange the matrix ${\rm \bf V}^{\delta}$ into a long vector and define the $N^2$-by-$P$ matrix ${\rm \bf S}$  whose $k$th-column stores the matrix ${\rm \bf S}_k$:
		\[{\rm Vec}^{\delta}_{(i-1)N+j}:= ({\rm \bf V}^{\delta})_{i,j} \quad \mbox{ and } \quad {\rm \bf S}_{(i-1)N+j,k}:=({{\rm \bf S}_k})_{i,j}\]
		for $i,j = 1,\dots, N$ and $k = 1,\dots, P.$
		The minimization functional in (\ref{prb:4}) then becomes
		\[\| \sum_{k=1}^P a_k {\rm \bf S}_k -{\rm \bf V}^{\delta} \|^2_F = \left\| {\rm \bf S} {\bf a} -{\rm Vec}^{\delta} \right\|_2^2\]
		where ${\bf a}=(a_1,\dots, a_P)^\top$. 
		And we end up with the following quadratic minimization problem under box constraints:
		\begin{eqnarray*}
			\min \left\{\| {\rm \bf S} {\bf a}-{\rm Vec}^{\delta} \|^2_2: \,a_k \in \mathbb{R}, \, 0 \le a_k \le \min({\it a},\beta^{\delta,N}_k) \right\}.
		\end{eqnarray*}
		
		\subsection{Numerical experiments}
\label{subsec:num}

In our numerical experiments, we plot the support of the minimizer $x \mapsto \sum_{k=1}^P \hat{a}_k\chi_{P_k}$ of the minimization problem (\ref{prb:4}), where $\{P_k\}$ is a partition of the unit disk $\Omega$ in $\R^2$ centered at the origin and is chosen independently of the finite element mesh that is used for solving the forward problems.

In the first experiment, the true inclusion $D$ includes a small ball $B$ centered at $(-0.4,-0.5)$ radius $0.1$, a rectangle $R$ whose lower-left corner is located at $(0.3,-0.65)$ and upper-right corner is located at $(0.45,-0.4)$ and an ellipse $E$ centered at $(0.1,0.4)$, horizontal semi-axis $0.3$, vertical semi-axis $0.1$. The reference conductivity $\sigma_0$ is assumed to be $1$ on $\Omega$. The true conductivity $\sigma$ is assumed to be $4$ on $B$, $3$ on $E$, $2$ on $R$ and $1$ outside $D$; and is plotted in the first picture of Figure \ref{fig:1}. The next three pictures of Figure \ref{fig:1} show the reconstruction images (aka. the support of the minimizer $x \mapsto \sum_{k=1}^P \hat{a}_k\chi_{P_k}$ of the minimization problem (\ref{prb:4})) of our method with respect to different levels of noise. 
Figure \ref{fig:1} yields that minimizing the residual under constraint $0 \le \kappa \le \sum_{k=1}^P \min({\it a},\beta^{\delta}_k) \chi_k$ is not affected much by the noise. Moreover, our method produces no artifacts even under high levels of noise. In Figure \ref{fig:2}, we consider the reconstruction images of the minimization residual problem under different constraints, to see that the upper bound ${\it a}$ is essential not only for preventing infinity upper bounds when $\beta^{\delta}_k=\infty$, but also for guaranteeing a good shape reconstruction (figures \ref{fig:1}d and \ref{fig:2}b). In many applications, a bound for  the conductivity change $\gamma$ is known; hence, in these cases, the value of ${\it a}:=1-\frac{1}{1+\inf_D \gamma}$ can be calculated a-priori. All of the reconstruction images showed in Figure \ref{fig:1} and Figure \ref{fig:2} are obtained by using \texttt{cvx} to minimize the residual (\ref{prb:4}).

In the second experiment, we consider the true inclusion $D$ as a ball centered at the origin and with radius $0.1$. The reference conductivity $\sigma_0$ is also assumed to be $1$ on $\Omega$. The true conductivity $\sigma$ is $4$ on $D$ and $1$ outside; and is plotted in the first picture of Figure \ref{fig:3}. The next two pictures of Figure \ref{fig:3} plot the support of the minimizer $x \mapsto \sum_{k=1}^P \hat{a}_k\chi_{P_k}$ of the minimization problem (\ref{prb:4}) following the \texttt{cvx} minimization algorithm when the $\beta^\delta_k$'s do and do not involve in the box constraint. A cut through the origin of three pictures of Figure \ref{fig:3} via the $x$-axis is presented in Figure  \ref{fig:4}. These pictures show that the upper bound $\beta^{\delta}_k$'s play an important role in the proof of the theoretical part and makes the reconstructions slightly better when the noise level is very low (figures \ref{fig:3} and \ref{fig:4}). 

Figure \ref{fig:5} shows the reconstruction images of the true conductivity in the first experiment (aka. the first picture of Figure \ref{fig:1}) when different minimization algorithms are applied to solve the minimization (\ref{prb:4}).  Under the constraint $0 \le \kappa \le \sum_{k=1}^P \min({\it a},\beta^\delta_k)\chi_k$ and $5\%$ relative noise, the reconstruction image using \texttt{cvx} looks perfectly well and contains no ringing artifact at all; while the \texttt{MATLAB} built-in function \texttt{quadprog} with \texttt{trust-region-reflective} Algorithm also yields very good result that reduces a lot of ringing artifacts compared with the standard Tikhonov approach (the third picture from the left of Figure \ref{fig:2}). However, the \texttt{MATLAB} built-in function \texttt{quadprog} with default option (in this case \texttt{interior-point-convex} Algorithm) totally fails to produce an approximation of the true conductivity. The minimum values and the amount of time taken when solving the minimization problem (\ref{prb:4}) to produce pictures in Figure \ref{fig:5} are shown in Table 1. We believe that the \texttt{trust-region-reflective} Algorithm can be optimized for real-time implementation.

\begin{table}\label{tab:1}
\caption{\label{arttype} Minimum value and runtime of pictures in Figure \ref{fig:5}.}
\footnotesize\rm
\begin{tabular*}{\textwidth}{@{}l*{15}{@{\extracolsep{0pt plus12pt}}l}}
\br
Algorithm&Minimum value&Runtime (second)\\
\mr
\texttt{cvx}&0.0126& 4818.4263\\
\texttt{trust-region-reflective}&0.0131&235.2102\\
\texttt{interior-point-convex}&0.2439&67.2561\\
\br
\end{tabular*}
\end{table}

The numerical experiments confirm that minimizing the residual of the linearized EIT equation under the constraint $0 \le \kappa \le \sum_{k=1}^P \min({\it a},\beta_k) \chi_k$ yields good approximations to the true conductivity change. The algorithm yields no artifacts and produces good shape reconstructions even under high levels of noise.  We also expect the same results for the complete electrode model setting.

All of the above arguments hold if we replace the Frobenius norm in (\ref{29.10.1}) by the operator norm. However, the numerical results with respect to the Frobenius norm are nicer (the last picture of Figures \ref{fig:2} and the first picture of Figure \ref{fig:5}).




\section{Conclusions}\label{Sec:con}
The popularly used reconstruction methods based on minimizing the usual linearized EIT equation are simple and fast for real-time implementation and produce good reconstruction images. However, these methods have no rigorous convergence results, and the reconstruction images usually contain ringing artifacts. On the other hand, monotonicity-based methods allow globally convergent implementation but usually produce bad images under high levels of noise or when real data are used. Our method is a combination of the usual minimization problem of the linearized EIT equation and the monotonicity-based method, which inherits most of the good properties of these two methods, such as stability under high noise and rigorous global convergence property. {{Besides, if the lower and upper bounds of the conductivity are known, all parameters of our method can be calculated a-priori.} }Moreover, to the best of our knowledge, this is the first reconstruction method based on minimizing the residual which has a rigorous global convergence property. However, we would admit that our method requires the definiteness assumption, that is, the true conductivity should be either always bigger  or always smaller than the reference conductivity over all the reference body. 

 In this paper, we establish rigorously theoretical results of our method and provide a few numerical experiments in an idealistic setting, i.e. the continuum model setting. In future works, we will apply this method to more realistic models such as the shunt model or the complete electrode model, which are commonly used models in practice.


\ack
The authors would like to thank Marcel Ullrich, Dominik Garmatter and Janosch Rieger for many fruitful discussions. MNM is also grateful to Li-Chang Hung and Professor Giovanni Alberti for their constant support. 

\section*{References}


\end{document}